\documentclass{article}
\usepackage[latin1]{inputenc}
\usepackage{amssymb}
\usepackage{verbatim}
\usepackage{array}
\usepackage{latexsym}
\usepackage{enumerate}
\usepackage{amsmath}
\usepackage{amsfonts}
\usepackage{amsthm}
\usepackage{color}
\usepackage[english]{babel}

\newtheorem{theorem}{Theorem}[section]
\newtheorem{proposition}[theorem]{Proposition}
\newtheorem{lemma}[theorem]{Lemma}
\newtheorem{corollary}[theorem]{Corollary}

\newtheorem{rem}[theorem]{Remark}

\def\cC{\mathcal C}

\def\cL{\mathcal L}

\def\fq{{\mathbb F}_q}
\def\F{\mathbb F}

\def\ord{\mbox{\rm ord}}

\def\gg{\mathfrak{g}}

\newcommand{\vp}{\varphi}
\newcommand{\xx}{\mathcal X}
\newcommand{\yy}{\mathcal Y}

\newcommand{\aut}{\mbox{\rm Aut}}

\newcommand{\kk}{\mathbb K}

\newcommand{\oo}{\mathcal O}
\newcommand{\vv}{\mathbf{v}}

\newcommand{\ga}{\alpha}

\newcommand{\ha}{{\textstyle\frac{1}{2}}}

\title{Curves containing all points of a finite projective Galois plane}
\date{}
\author{Gregory Duran Cunha}
\begin{document}
\maketitle


    \begin{abstract} In the projective plane $PG(2,q)$ over a finite field of order $q$, a Tallini curve is a
     plane irreducible (algebraic) curve of (minimum) degree $q+2$ containing all points of $PG(2,q)$. Such curves
     were investigated by G.~Tallini \cite{T1,T2} in 1961, and by Homma and Kim \cite{homma} in 2013. Our results concern the automorphism groups, the Weierstrass semigroups, the Hasse-Witt invariants, and quotient curves of the Tallini curves.
     \end{abstract}

\providecommand{\keywords}[1]{\textbf{\textit{Keywords:}} #1}


 \section{Introduction}

A fundamental result on plane irreducible (algebraic) curves defined over a finite field $\mathbb{F}_q$ is the Hasse-Weil bound
$$S_q\leq q+1+(n-1)(n-2)\sqrt{q}$$
where $n$ is the degree of the curve and $S_q$ is the number of its points lying in the projective plane $PG(2,q)$ of order $q$; see \cite[Section 9.6]{hirschfeld-korchmaros-torres2008}.
Any plane (possibly reducible) curve containing all points of $PG(2,q)$ has degree at least $q+1$, and if equality holds then the curve splits into the $q+1$ lines of a pencil in $PG(2,q)$. A complete classification of plane  irreducible curves of degree $q+2$ containing all points of $PG(2,q)$ was given by G.~Tallini \cite{T1,T2}; see also \cite{AK}, and \cite{homma}. Up to projective transformations in $PG(2,q)$, each such curve $\xx$ has homogeneous equation of type
\begin{equation}\label{1}
(aX_0 + bX_1 + cX_2) \vp_{01} - X_0 \vp_{02} + X_2 \vp_{12} = 0,
\end{equation}
where $\vp_{ij} = X_i^qX_j - X_iX_j^q$ and $a,b,c$ are elements in $\fq$ such that the cubic equation
\begin{equation}\label{eq2}
X^3 - cX^2 - aX - b = 0
\end{equation}
is irreducible over $\fq$. In this paper, the above irreducible curve $\xx$ of degree $n=q+2$ is named \emph{Tallini curve}.

G.~Tallini proved that $\xx$ has no singular points in $PG(2,q)$. Homma and Kim \cite[Section 3]{homma} extended his result to any point in $PG(2,\kk)$ where $\kk$ is the algebraic closure of $\mathbb{F}_q$. Therefore, $\xx$ is a plane nonsingular curve of genus $\gg=\ha (n-1)(n-2)=\ha q(q+1)$.

 G.~Tallini showed that the automorphism group $G_q$ of $\xx$ over $\mathbb{F}_q$ contains a Singer cycle, that is, a cyclic subgroup $S$ of $PGL(3,q)$ of order $q^2+q+1$ acting on $PG(2,q)$ as a regular permutation group. He also claimed that  $G_q$ may be a bit larger but only for some special curves, named harmonic and equianharmonic curves in \cite{T1,T2}.  More precisely, Homma and Kim proved \cite[Theorem 5.4]{homma} that if $G_q$ with $q>2$ is larger than $S$ then $G_q$ is the normalizer of $S$ in $PGL(3,q)$, that is, $G_q=S\rtimes C_3$, the semidirect product of $S$ by a group $C_3$ of order $3$.

In this paper we go on with the study of the Tallini curves, also from the function field point of view. We look at the Tallini curves in the projective plane $PG(2,\kk)$ defined over the algebraic closure $\kk$ of $\fq$. Our Theorem
 \ref{20jun2016} shows that up to projective equivalence in $PG(2,\kk)$, the Tallini curve $\xx$ is projectively equivalent to the curve 
 \begin{equation}
\label{3}
X_1^{q+1}X_2+X_2^{q+1}X_0+X_0^{q+1}X_1=0.
\end{equation}
For $q=2$, $\xx_q$ is the famous plane Klein quartic whose automorphism group is isomorphic to $PSL(2,7)$. 
We mention that the curve $\xx_q$ was first investigated in \cite{pell}, and we refer to it as the \emph{Pellikaan curve}.
 From the proof of Theorem \ref{20jun2016}, the smallest projective plane $PG(2,q^{3i})$ where this equivalency occurs is in general much larger than $PG(2,q^3)$ as $\mathbb{F}_{q^{3i}}$
 turns out to be the smallest overfield of $\mathbb{F}_{q^3}$ containing the roots of the equation
 $X^{q^2+q+1} = (\ga^q - \ga)^{q^2+q-2}$ where $\ga$ is a root of (\ref{eq2}). Here $i$ divides $q^2+q+1$ and the automorphism group of $\xx$ in $PG(2,\kk)$ is isomorphic to $S\rtimes C_3$ where $S$ is defined over $\mathbb{F}_q$ but $C_3$ is in general defined over $\mathbb{F}_{q^{3i}}$.

 For every divisor $d$ of $q^2+q+1$, the curve $\xx_q$ has a quotient curve $\xx_q/C_d$ with respect to a cyclic group $C_d$ of order $d$. In case where $q$ is a square, that is, $q=p^{2i}$ with $p$ prime and $i\geq 1$, the factorization $p^{4i}+p^{2i}+1=(p^{2i}+p^i+1)(p^{2i}-p^i+1)$ raises the question whether the quotient curve $\xx_q/C_d$ with $d=p^{2i}-p^i+1$ is isomorphic to $\xx_{p^i}$. The answer is affirmative, see Theorem \ref{th220jun2016}.

 We also show that $\xx_p$ is an ordinary curve, that is, its genus $\gg=\ha p(p+1)$ coincides with its Hasse-Witt invariant. For this purpose, we prove that no exact differential of $\kk(\xx_p)$ is regular, and then use the properties of the Cartier operator to show that $\xx_p$ is ordinary. It should be noticed that this result does not hold true for $q>p$; see \cite{ms}.


\section{Irreducible curves of minimal degree containing all points of $PG(2,q)$}

Let $\xx$ be an irreducible plane curve of degree $q+2$ defined over $\fq$ containing all points of $PG(2,q)$. In his paper \cite{T1}, Tallini proved that such a curve exists and it has equation \eqref{1}.

 Now, we recall some facts from \cite{T1}. The polar net of $\xx$ is a net of conics and it is easy to see that this net is homaloidal with the three base points, namely $A_0 = (\ga_0 : 1 : \ga_0^2)$, $A_1 = (\ga_1 : 1 : \ga_1^2)$ and $A_2 = (\ga_2 : 1 : \ga_2^2)$ where $\ga_0, \ga_1, \ga_2$ are the three solutions of \eqref{eq2} in $\F_{q^3}$. In particular, the points $A_i$ are conjugate over $\F_{q^3}$. Furthermore, they belong to $\xx$ and they are simple points for it. We call each of these three points a base point of $\xx$.

Also in \cite{T1}, Tallini showed that the automorphism group of $\xx$ contains a Singer cycle, that is, a cyclic subgroup $S$ of $PGL(3,q)$ of order $q^2+q+1$ acting on $PG(2,q)$ as a regular permutation group. Moreover, $S$, viewed as a subgroup of $PGL(3,q^3)$, fixes $A_0$, $A_1$, $A_2$.

The following result was proved in \cite{homma}.

\begin{theorem} \label{t2}
The curve $\xx$ is non-singular.
\end{theorem}

\begin{proof}
If $P$ is a singular point of $\xx$ then its orbit under $S=  \langle \phi \rangle$
$$
\oo = \{  \phi^i(P)  \; | \;  0 \leq i \leq q^2 + q  \}
$$
consists of singular points of $\xx$. Note that the size of $\oo$ is $1$ or $q^2+q+1$. Since the number of singular points of $\xx$ is at most $q(q+1)/2$, then $\oo = \{ P \}$. This yields that every singular point of $\xx$ is fixed by $S$. Since $S$ fixes only $A_0$, $A_1$ and $A_2$, which are simple points, it follows that $\xx$ has no singular points.
\end{proof}


\label{3}

\begin{theorem}
\label{20jun2016} Any Tallini curve is projectively equivalent to the Pellikaan curve  
over $\F_{q^{3(q^2+q+1)}}$.
\end{theorem}

\begin{proof}
Let $\ga_0, \ga_1, \ga_2 \in \F_{q^3}$ be the distinct solutions of \eqref{eq2}, and let $\yy$ be the image of $\xx$ under the linear map associated to the non-singular matrix
$$
M = \left(
\begin{array}{ccc}
 \ga_0 &  \ga_1  &  \ga_2  \\
 1 & 1  & 1  \\
  \ga_0^2 & \ga_1^2  &\ga_2^2
\end{array}
\right).
$$
That is, $\yy$ is the curve given by $G(X_0,X_1,X_2)=0$, where
\begin{equation}
\label{tallini_image}
G = F  \Big(  \sum_{i=0}^2 \ga_iX_i , \sum_{i=0}^2 X_i, \sum_{i=0}^2  \ga_i^2X_i  \Big),
\end{equation}
and $\xx = \vv(F)$ is the Tallini curve. A straightforward computation gives
$$
G = c_{01} X_0^{q+1}X_1 + c_{02} X_0^{q+1}X_2 + c_{10} X_1^{q+1}X_0 + c_{12} X_1^{q+1}X_2 + c_{20} X_2^{q+1}X_0 + c_{21} X_2^{q+1}X_1,
$$
where
$$
c_{ij} = (\ga_i - \ga_j)^2(\ga_i^q - \ga_i)(\ga_j - \ga_i^q),  \;  \;  \text{ for }  \;  0 \leq i,j \leq 2.
$$
Note that $c_{ij} = 0$ whenever $\ga_j = \ga_i^q$. Since the Frobenius map acts transitively on  $\{ \ga_0 ,  \ga_1  ,  \ga_2 \}$, it follows that either $(\ga_0 ,  \ga_1  ,  \ga_2) = (\ga_1^q ,  \ga_2^q  ,  \ga_0^q)$ or $(\ga_0 ,  \ga_1  ,  \ga_2) = (\ga_2^q ,  \ga_0^q  ,  \ga_1^q)$. In the former case, we have $c_{10} = c_{02} = c_{21} = 0$ and then
\begin{equation}
\label{G1}
G = c_{01} X_0^{q+1}X_1  +  c_{12} X_1^{q+1}X_2 + c_{20} X_2^{q+1}X_0,
\end{equation}
whereas the latter case gives $c_{01} = c_{12} = c_{20} = 0$ and
\begin{equation}
\label{G2}
G =  c_{02} X_0^{q+1}X_2 + c_{10} X_1^{q+1}X_0 +  c_{21} X_2^{q+1}X_1.
\end{equation}
We prove the result in the case $G$ is given by \eqref{G1}, and then case \eqref{G2} will follow analogously. Note that from $(\ga_0 ,  \ga_1  ,  \ga_2) = (\ga_1^q ,  \ga_2^q  ,  \ga_0^q)$, equation \eqref{G1} can be written as
\begin{equation}
\label{GG}
(\ga_1^q - \ga_1) X_0^{q+1}X_2 + (\ga_1^q - \ga_1)^qX_1^{q+1}X_0 +  (\ga_1 - \ga_1^{q^2}) X_2^{q+1}X_1 = 0.
\end{equation}
Finally, one can easily check that the curve given by \eqref{GG}  is the image of $\xx_q$ under the transformation
$$
(X_0,X_1,X_2)  \mapsto (\mu X_0, \lambda X_1,X_2)
$$
where $\mu,\lambda \in \F_{q^{3(q^2+q+1)}}$ satisfy
$$
\lambda^{q^2+q+1} = \frac{(\ga_1^q - \ga_1)^3}{(\ga_1^q - \ga_1)^{q^2+q+1}}  \;  \text{ and } \; \mu =  \frac{\lambda^{q+1}(\ga_1^{q^2} - \ga_1^q)}{\ga_1 - \ga_1^{q^2}}.
$$
This finishes the proof.
\end{proof}

\section{Weierstrass semigroup at a base point}

Let $\Sigma =\kk(x,y)$, with $xy^{q+1} + x^{q+1} + y = 0$, be the function field of the Pellikaan curve $\xx_q$ and consider the fundamental triangle
$$
O = (0:0:1), \; X_{\infty} = (1:0:0), \; Y_{\infty} = (0:1:0).
$$
The tangent lines to $\xx_q$ at $O$, $X_{\infty}$ and $Y_{\infty}$ are $l_Y = \vv(Y)$, $l_Z = \vv(Z)$ and $l_X = \vv(X)$, respectively. Note that the points in $l_Y \cap \xx_q$ are $O$ and $X_{\infty}$ with
$$
I(O,l_Y \cap \xx_q) = q+1 \text{ and }  I(X_{\infty},l_Y \cap \xx_q) =1,
$$
the points in the intersection $l_Z \cap \xx_q$ are $X_{\infty}$ and $Y_{\infty}$ with
$$
I(X_{\infty},l_Z \cap \xx_q) = q+1 \text{ and }  I(Y_{\infty},l_Z \cap \xx_q) =1,
$$
and the points in the intersection $l_X \cap \xx_q$ are $Y_{\infty}$ and $O$ with
$$
I(Y_{\infty},l_X \cap \xx_q) = q+1 \text{ and }  I(O,l_X \cap \xx_q) =1.
$$

From \cite[Theorem 6.42]{hirschfeld-korchmaros-torres2008} the principal divisor of $x$ is given by
$$
\begin{array}{ccl}
(x) &  = &  l_X \cdot \xx_q - l_Z \cdot \xx_q \\
  &  = &  I(O,l_X \cap \xx_q) O + I(Y_{\infty},l_X \cap \xx_q) Y_{\infty} - I(X_{\infty},l_Z \cap \xx_q)X_{\infty} - I(Y_{\infty},l_Z \cap \xx_q)Y_{\infty} \\
  &  = & O + (q+1)Y_{\infty} - (q+1)X_{\infty} - Y_{\infty} \\
  & = & O + qY_{\infty} - (q+1)X_{\infty}
\end{array}
$$
\noindent and the principal divisor of $y$ is given by
$$
\begin{array}{ccl}
(y) &  = &  l_Y \cdot \xx_q - l_Z \cdot \xx_q \\
  &  = &  I(O,l_Y \cap \xx_q) O + I(X_{\infty},l_Y \cap \xx_q) X_{\infty} - I(X_{\infty},l_Z \cap \xx_q)X_{\infty} - I(Y_{\infty},l_Z \cap \xx_q)Y_{\infty} \\
  &  = & (q+1)O + X_{\infty} - (q+1)X_{\infty} - Y_{\infty} \\
  & = & (q+1)O  - qX_{\infty} -Y_{\infty}.
\end{array}
$$
\noindent
For $1 \leq n \leq q+1$, the above relations gives
$$
\left( \dfrac{y}{x^n}  \right) = (q+1 - n) O + (n(q+1)-q)X_{\infty} - (nq+1)Y_{\infty},
$$
hence the divisor of poles of $y/x^n$ is
$$
\left( \dfrac{y}{x^n}  \right)_{\infty} = (nq+1)Y_{\infty},  \; \;  \text{ for }  1 \leq n \leq q+1.
$$

\begin{theorem}
The Weierstrass semigroup at a base point of $\xx_q$ is the semigroup generated by $q+1, 2q+1, \ldots, (q+1)q+1$.
\end{theorem}

\begin{proof}
We may assume that the base point is $Y_{\infty}$. From the above discussion, $q+1, 2q+1, \ldots, (q+1)q+1$ belong to the Weierstrass semigroup $H(Y_{\infty})$. Let $H= \langle q+1, 2q+1, \ldots, (q+1)q+1  \rangle$ be the semigroup generated by $ q+1, 2q+1, \ldots, (q+1)q+1$. Since $H \subset H(Y_{\infty})$ and the number of gaps in $H$ is $q(q+1)/2$, which is the number of gaps in $H(Y_{\infty})$, the assertion follows.
\end{proof}

\begin{theorem}
If the function field of $\xx_q$ is given by $\kk(x,y)$, with $xy^{q+1} + x^{q+1} + y = 0$, then the divisor of the differential $dx$ is
$$
(dx) = (q^2 + 2q)Y_{\infty} - (q+2)X_{\infty}.
$$
\end{theorem}

\begin{proof}
The curve $\xx_q$ is constituted by two points in the infinity $X_{\infty}$, $Y_{\infty}$ and the affine points $P=(a:b:1)$, with $ab^{q+1} + a^{q+1} + b =0$. The tangent line to $\xx_q$ at an affine point $P=(a,b)$ is not vertical, in fact, suppose by contradiction that $X=0$ is the tangent to $\xx_q = \vv(F)$ at $P$. Then $\frac{\partial F}{\partial Y} = 0$ in the point $P = (a,b)$. It means that $ab^q + 1 = 0$ and hence $a \neq 0$. On the other hand, $ab^{q+1} + a^{q+1} + b =0$ becomes $-b + a^{q+1} + b =0$, and therefore $a=0$, a contradiction.
Thus, a primitive parametrization of $\xx_q$ at $P$ is given by
$$
x(t) = a + t,
$$
$$
y(t) = b_0 + b_1t + \cdots, \; \; b_0 = b.
$$
Therefore, $\ord_{P} dx = 0$, for all affine point $P$ in $\xx_q$. So it follows that $(dx) = n X_{\infty} + m Y_{\infty}$ with $n + m = 2g-2$, where $g = q(q+1)/2$ is the genus of the curve $\xx_q$. Since $(x) = qY_{\infty} + O - (q+1)X_{\infty}$ and $q+1$ is not divisible by $p$,  $\ord_{X_\infty} dx = -(q+2)$, that is, $n = -(q+2)$. From $n+m = 2g-2$ it follows that $m = q^2 + 2q$.
\end{proof}

\section{The automorphism group}

Again, let $\kk(\xx_q)= \kk(x,y)$, with $xy^{q+1} + x^{q+1} + y = 0$, be the function field of the Pellikaan curve $\xx_q$ and let $\lambda \in \kk$  be a primitive $(q^2+q+1)$-root of unity and define the linear collineations
$$
\sigma : (X_0,X_1,X_2) \mapsto (X_0,\lambda X_1, \lambda^{q+1} X_2).
$$
and
$$
\tau : (X_0,X_1,X_2) \mapsto (X_2,X_0,X_1).
$$
It is straightforward to see that the automorphism group of $\xx_q$ in $PG(2,\kk)$ contains $S \rtimes C_3$ as a subgroup, where $S= \langle \sigma \rangle$ and $C_3 = \langle \tau \rangle$.

Observe that $S$ has order $q^2+q+1$. Therefore, going back to the original equation \eqref{1} of $\xx$, $S$ becomes a Singer cycle of $PG(2,q)$.

To show that $S \rtimes C_3$ is actually the whole automorphism group of $\xx_q$ over $\kk$, we need some lemmas.

\begin{lemma}\label{9}
If $Q$ is a point in $\xx_q$ and $t_Q$ is its tangent line, then
$$
I(Q, \xx_q \cap t_Q) =
\left\{
\begin{array}{ccc}
 q+1, &  \text{if}  & Q \in \{  X_{\infty}, Y_{\infty}, O \}  \\
  2, &  \text{if}  &   Q  \not\in \{  X_{\infty}, Y_{\infty}, O \}.
\end{array}
\right.
$$
\end{lemma}

\begin{proof}
We have this result for the points in the fundamental triangle, see section $3$. Suppose $Q=(a,b)$ with $ab^{q+1}+a^{q+1}+b=0$ and $ab \neq 0$. The tangent line $t_Q$ to $\xx_q$ at $Q$ is given by
$$
T(X,Y) = \frac{b}{a}X + \frac{a^{q+1}}{b}Y + ab^{q+1} = 0.
$$
A primitive parametrization of $\xx_q$ at $Q$ is given by
$$
x(t) = a + t
$$
$$
y(t) = b - \frac{b^2}{a^{q+2}}t - \frac{b^{q+3}}{a^{2q+3}}t^2 - \frac{b^{2q+4}}{a^{3q+4}}t^3 - \cdots.
$$
Therefore $T(x(t),y(t)) =  - (b/a)^{q+2}t^2 + \cdots$ has order $2$, that is, $I(Q, \xx_q \cap t_Q)=2$.
\end{proof}

\begin{lemma}\label{10}
Every automorphism in $\aut(\xx_q)$ preserves the triangle $\{  X_{\infty}, Y_{\infty}, O \}$.
\end{lemma}

\begin{proof}
Let $\ga \in \aut(\xx_q)$ and suppose $\ga: P \mapsto Q$ with $P \in \{  X_{\infty}, Y_{\infty}, O \}$ and $Q \notin \{  X_{\infty}, Y_{\infty}, O \}$. Consider the lines $t_Q$ and $l_Q$ given by
$$
t_Q :  \;  T(X,Y)=0
$$
$$
l_Q: \; L(X,Y) = 0
$$
such that $t_Q$ is the tangent line to $\xx_q$ at $Q$ and $l_Q$ is a secant line through $Q$. Consider the curve $\cC$ of degree $q-1$ given by
$$
T(X,Y)L(X,Y)^{q-2}=0.
$$
By the Lemma \ref{9},
$$
I(Q, \xx_q \cap \cC) = I(Q, \xx_q \cap t_Q) + (q-2)I(Q, \xx_q \cap l_Q) = 2 + (q-2) = q.
$$
Observe that $W := \xx_q \cdot \cC$ is a canonical divisor such that $\cL(W-qQ) \neq \cL(W-(q+1)Q)$. Thus, by Riemann-Roch Theorem, $\ell((q+1)Q) = \ell(qQ)$. Hence $q+1$ is a gap number at $Q$, but this is a contradiction as $q+1$ is a non-gap at $P$.
\end{proof}

\begin{theorem}
\label{th120jun2016}
$\aut(\xx_q) = S \rtimes C_3$.
\end{theorem}

\begin{proof}
Let $\ga \in \aut(\xx_q)$. Since $\xx_q$ is non-singular, $\ga$ can be represented as a matrix $A$ in $PGL(3,\kk)$. By the Lemma \ref{10}, $\ga$ preserves the fundamental triangle. First, suppose that $\ga$ fixes all vertices of the fundamental triangle, then
$$
A =
\left(
\begin{array}{ccc}
\xi  & 0  & 0  \\
 0 & \eta  &  0 \\
  0& 0  &   1
\end{array}
\right).
$$
Since $\ga$ preserves $\xx_q$,
$$
\xi \eta^{q+1} xy^{q+1} + \xi^{q+1}x^{q+1} + \eta y = 0
$$
in $\kk(\xx_q)= \kk(x,y)$. Hence $\eta = \xi^{q+1}$ and $\xi^{q^2+q+1}=1$. Therefore $\ga \in S$.
Now, suppose that $\ga$ fixes no vertices of the fundamental triangle. In that case, $\ga = \tau$ or $\ga = \tau^2$. To complete the proof we only need to show that the case when $\ga$ fixes only one point in the fundamental triangle does not happen. Suppose that $\ga$ only fixes the origin, thus $\ga$ interchanges $X_{\infty}$ and $Y_{\infty}$. Hence,
$$
A =
\left(
\begin{array}{ccc}
0 & \xi  & 0  \\
 \eta & 0  &  0 \\
  0& 0  &   1
\end{array}
\right).
$$
where $\xi, \eta \in \kk$. Since $\ga$ preserves $\xx_q$,
$$
\xi \eta^{q+1} yx^{q+1} + \xi^{q+1}y^{q+1} + \eta x = 0
$$
in $\kk(x,y)$. It means that, there exists $c\neq 0$ in $\kk$ such that
$$
\xi \eta^{q+1} YX^{q+1} + \xi^{q+1}Y^{q+1} + \eta X = c(XY^{q+1} + X^{q+1} + Y),
$$
which is a contradiction. The cases when $\ga$ fixes only $X_{\infty}$ or $Y_{\infty}$ are analogues.
\end{proof}
Theorems \ref{th120jun2016} and \ref{20jun2016} have the following corollary.
\begin{corollary}
\label{cor20jun2016}
The automorphism group of $\xx$ is $S\rtimes C_3$, where $S$ is defined over $\mathbb{F}_{q}$ but $C_3$ is defined over $\mathbb{F}_{q^i}$ with $i=3(q^2+q+1)$.
\end{corollary}
\begin{rem}{\em{
By a result of Cossidente and Siciliano, see \cite{CS}, and \cite[Theorem 11.110]{hirschfeld-korchmaros-torres2008}, if  a plane nonsingular curve ${\mathcal{C}}$ of $PG(2,q)$ of degree $q+2$ has an automorphism group $S\rtimes C_3$ where $S$ is a Singer cycle of $PG(2,q)$ then ${\mathcal{C}}$ is projectively equivalent to the Pellikaan curve $\xx_q$ where equivalency is meant in $PG(3,q^3)$. It should be noted that the authors in \cite{CS} claimed their result was valid under a weaker condition, namely when the automorphism group of ${\mathcal{C}}$ contains $S$. But this turns out incorrect
by Theorem \ref{20jun2016} 
.
}}
\end{rem}

\section{Quotient curves}

Suppose that $q$ is a square, say $q=p^{2i}$, $i \geq 1$. Thus $q^2+q+1 = (p^{2i}+p^i+1)(p^{2i}-p^i+1)$. Let $\lambda$ be a primitive $(q^2+q+1)$-root of unity in $\kk$. A straightforward computation shows that $\alpha$ defined by
$$
\alpha(x) = \lambda x,  \;  \;   \;  \alpha(y) = \lambda^{q+1} y,
$$
is a $\kk$-automorphism of $\kk(\xx_q)$ of order $q^2+q+1$, where $\kk(\xx_q)= \kk(x,y)$, with $xy^{q+1} + x^{q+1} + y = 0$, is the function field of $\xx_q$.

Let $h = \alpha^{p^{2i}+p^i+1}$. The group $H= \langle h \rangle$ has order $p^{2i}-p^i+1$. The next results provide equations for the quotient curve $\xx_q / H$ and one of those equations turns out to be
$$
XY^{p^i+1} + X^{p^i+1} + Y=0,
$$
which is indeed the equation of the Tallini curve $\xx_{p^i}$.

\begin{proposition}
The quotient curve $\xx_q / H$ is isomorphic to the curve given by the equation
$$
X^{p^{2i}+p^i+1}Y^{q+1} + Y +1 = 0.
$$
\end{proposition}

\begin{proof}
Take $\xi, \eta$ from $\kk(x,y)=\kk(\xx_q)$ given by
$$
\xi = x^{p^{2i}-p^i+1},  \;  \;   \;   \eta = x^{-(q+1)}y.
$$
Clearly $h(\xi) = \xi$ and $h(\eta) = \eta$, then $\kk(\xi, \eta) \subset \kk(x,y)^H$. Note that $\kk(x,y) = \kk(\xi,\eta)(x)$ and $T^{p^{2i}-p^i+1} - \xi$ is a polynomial in $\kk(\xi,\eta)[T]$ which has $x$ as a root. Thus $[\kk(x,y):\kk(\xi,\eta)] \leq p^{2i}-p^i+1$. Note that $[\kk(x,y):\kk(x,y)^H] = \ord(H) = p^{2i}-p^i+1$, hence $[\kk(x,y)^H : \kk(\xi, \eta)]=1$, therefore $\kk(x,y)^H = \kk(\xi, \eta)$.

Finally, since $xy^{q+1} + x^{q+1} + y = 0$ and $\eta = x^{-(q+1)}y$ we get
$$
x^{q^2 + 2q +2} \eta^{q+1} + x^{q+1} + x^{q+1}\eta = 0.
$$
Thus, $x^{q^2+q+1}\eta^{q+1} + 1 + \eta = 0$. Since $\xi = x^{p^{2i}-p^i+1}$ and $q^2+q+1 = (p^{2i}+p^i+1)(p^{2i}-p^i+1)$ we get
$$
\xi^{p^{2i}+p^i+1}\eta^{q+1} + \eta +1 = 0.
$$
\end{proof}

\begin{proposition}
The quotient curve $\xx_q / H$ is isomorphic to the curve given by the equation
$$
X^{p^{2i}+p^i+1} + Y^{p^i+1} + Y^{p^i} = 0.
$$
\end{proposition}

\begin{proof}
By the previous proposition, $\kk(\xx_q/H) = \kk(x,y)$, with
$$
x^{p^{2i}+p^i+1}y^{q+1} + y +1 = 0,
$$
that is,
$$
x^{p^{2i}+p^i+1} + \frac{1}{y^q} + \frac{1}{y^{q+1}} = 0.
$$
Putting $\xi = x$ and $\eta = 1/y$ gives
$$
\xi^{p^{2i}+p^i+1} + \eta^{q} + \eta^{q+1} = 0.
$$
Dividing by $\eta^{p^{2i}+p^i+1}$ and using $q=p^2$ gives
$$
\frac{\xi^{p^{2i}+p^i+1}}{\eta^{p^{2i}+p^i+1}} + \frac{1}{\eta^{p^i+1}} + \frac{1}{\eta^{p^i}} = 0.
$$
Replacing $u=\xi / \eta$ and $v= 1/\eta$ gives
$$
u^{p^{2i}+p^i+1} + v^{p^i+1} + v^{p^i} = 0.
$$
\end{proof}

\begin{theorem}
\label{th220jun2016}
The quotient curve $\xx_q / H$ is isomorphic to the curve $\xx_{p^i}$ given by the equation
$$
XY^{p^i+1} + X^{p^i+1} + Y =0.
$$
\end{theorem}

\begin{proof}
Consider the function field $\kk(\xx_{p^i}) = \kk(x,y)$, with $xy^{p^i+1} + x^{p^i+1} + y =0$. We have that
$$
x(y^{p^i+1} + x^{p^i}) + y =0.
$$
Raising to the $p^i$-th power gives
$$
x^{p^i}(y^{p^i+1} + x^{p^i})^{p^i} + y^{p^i} = 0.
$$
Multiplying by $x$ gives
$$
x^{p^i+1}(y^{p^i+1} + x^{p^i})^{p^i} + xy^{p^i} = 0.
$$
This also can be written as,
$$
x^{p^{2i}+p^i+1} + (-xy^{p^i} -1)^{p^i}(-xy^{p^i}) = 0.
$$
Putting $u=x$ and $v= -xy^{p^i}-1$ gives
$$
u^{p^{2i}+p^i+1} + v^{p^i+1} + v^{p^i} =0.
$$
Note that $\kk(u,v) = \kk(x,y^{p^i}) \subset \kk(x,y)$. Since $xy^{p^i+1} + x^{p^i+1} + y =0$,
$$
y = - \frac{x^{p^i+1}}{xy^{p^i} + 1},
$$
that is, $y$ belongs to $\kk(x,y^{p^i}) $. Therefore, $\kk(u,v)=\kk(x,y)$.
\end{proof}

\section{The Hasse-Witt invariant}

In this section $q=p$ is a prime number. Let $\Sigma =\kk(x,y)$, with $xy^{p+1} + x^{p+1} + y = 0$, be the function field of the Pellikaan curve $\xx_p$, $\gg$ its genus and $\gamma$ its Hasse-Witt invariant. The partial derivative of $F(X,Y) = XY^{p+1} + X^{p+1} + Y$ with respect to $Y$ is $F_Y(X,Y) = XY^p +1$.

Consider $\Delta_{\Sigma} = \{  udx  \; | \;  u \in \Sigma\}$ the differential module of $\Sigma$ and $C : \Delta_{\Sigma}^{(1)} \rightarrow \Delta_{\Sigma}^{(1)}$ the Cartier operator defined on the space of holomorphic differentials
$$
\Delta_{\Sigma}^{(1)} = \{  w \in \Delta_{\Sigma}  \; | \;  (w) \geq 0 \}.
$$

\begin{theorem}
The Hasse-Witt invariant of $\xx_p$ is equal to its genus.
\end{theorem}

\begin{proof}
Let $w$ be an exact differential in $\Delta_{\Sigma}^{(1)}$, that is, $C(w)=0$. Then $w$ can be written in the form
$$
w = (u_1^p + u_2^px + \cdots + u_{p-2}^px^{p-2}) dx.
$$
From \cite{Gorenstein}, a basis for the $\kk$-vector space $\Delta_{\Sigma}^{(1)}$ is given by
$$
\mathfrak{B} = \left\{  \frac{x^i y^j}{F_Y(x,y)}dx  \;  \vert  \;  0 \leq i+j  \leq p-1 \right\}.
$$
Thus
$$
u_1^p + u_2^px + \cdots + u_{p-2}^px^{p-2}  =   \frac{u(x,y)}{F_Y(x,y)}
$$
where $u(X,Y)$ is a polynomial in $\kk[X,Y]$ of degree at most $p-1$. Let
$$
u(x,y) = \sum_{i+j \leq p-1 } a_{ij} x^i y^j.
$$
Since $xy^{p+1} + x^{p+1} + y=0$ then
$$
\frac{1}{y} = - \frac{xy^p+1}{x^{p+1}}.
$$
Thus we have
\begin{equation*}
     \begin{aligned}
      \frac{u(x,y)}{F_Y(x,y)}
      & = \frac{u(x,y)}{xy^p+1} \\
      & = \frac{1}{xy^p+1} \sum_{i+j \leq p-1 } a_{ij} x^i y^j \\
      & = \frac{y^p}{xy^p+1} \sum_{i+j \leq p-1 } a_{ij} x^i  \left(  \frac{1}{y}  \right)^{p-j}   \\
      & = \frac{y^p}{xy^p+1} \sum_{i+j \leq p-1 } a_{ij} x^i  \left( - \frac{xy^p+1}{x^{p+1}}  \right)^{p-j} \\
      & = \sum_{i+j \leq p-1 } (-1)^{j+1} a_{ij} \frac{y^p}{x^{p^2+p-jp}} x^{i+j} (xy^p+1)^{p-1-j}
     \end{aligned}
\end{equation*}
\begin{equation*}
     \begin{aligned}
       & = \sum_{i+j \leq p-1 } (-1)^{j+1} a_{ij} \frac{y^p}{x^{p^2+p-jp}} x^{i+j} \sum_{k=0}^{p-1-j} \binom{p-1-j}{k} x^ky^{kp} \\
       & = \sum_{i+j \leq p-1 } \sum_{k=0}^{p-1-j} (-1)^{j+1} \binom{p-1-j}{k} a_{ij}(x^{jp-p^2-p} y^{(k+1)p}) x^{i+j+k} \\
       & = \sum_{i+j \leq p-1 } \sum_{k=0}^{p-1-j} w_{ijk}^p x^{i+j+k}
     \end{aligned}
\end{equation*}
where
$$
w_{ijk} = \left( (-1)^{j+1} \binom{p-1-j}{k} a_{ij} \right)^{1/p} x^{j-p-1} y^{k+1}.
$$
Hence,
$$
u_1^p + u_2^px + \cdots + u_{p-2}^px^{p-2} = \sum_{i+j \leq p-1 } \sum_{k=0}^{p-1-j} w_{ijk}^p x^{i+j+k}.
$$
The term of degree $p-1$ in $x$ on the right side is
$$
\left( \sum_{i+j \leq p-1 } w_{ijk_0}^p \right) x^{p-1}
$$
where $k_0 = p-1-j-i$. It means that
$$
 \sum_{i+j \leq p-1 } w_{ijk_0}^p = 0
$$
$$
\sum_{i+j \leq p-1 } w_{ijk_0} = 0
$$
$$
\sum_{i+j \leq p-1 } \left( (-1)^{j+1} \binom{p-1-j}{k_0} a_{ij} \right)^{1/p} x^{j-p-1} y^{k_0+1}=0
$$
$$
\sum_{i+j \leq p-1 } \left( (-1)^{j+1} \binom{p-1-j}{p-1-j-i} a_{ij} \right)^{1/p} x^{j-p-1} y^{p-(i+j)} =0
$$
$$
\sum_{i+j \leq p-1 } \left( (-1)^{j+1} \binom{p-1-j}{p-1-j-i} a_{ij} \right)^{1/p} x^{j} y^{p-(i+j)} =0.
$$

Since the last equation has degree at most $p$, it must be equal to zero. Thus all coefficients $a_{ij}$ are equal to zero, and therefore $u(x,y)=0$.
This shows that $w=0$, that is, the Cartier operator $C : \Delta_{\Sigma}^{(1)} \rightarrow \Delta_{\Sigma}^{(1)}$ has trivial kernel.

Let $V^0$ be the space of all $w \in \Delta_{\Sigma}^{(1)}$ such that $C^i(w)=0$ for some $i\geq 1$. Note that if $C^i(w)=0$, then $C^{i-1}(w) \in \ker(C)$. Hence $V^0 = \{ 0\}$. This implies that the Hasse-Witt matrix $(h_{ij})$ over $\kk$ of $C$ has maximum rank equal to $\gg$, and consequently the matrix
$$
M = (h_{ij})(h_{ij}^{p})\cdots(h_{ij}^{p^{g-1}})
$$
has rank $\gg$. Therefore, $\gamma = \gg$.
\end{proof}

\vspace{0,5cm}\noindent {\em Acknowledgment}:

\vspace{0.2cm}\noindent The author thanks G\'abor Korchm\'aros for his guidance and helpful discussions
about the topic of the present research which was carried out when the author was visiting the Universit\`a degli Studi della Basilicata (Italy) with a financial support of FAPESP-Brazil (grant 2015/10181-8).


\vspace{0,5cm}\noindent {\em Authors' addresses}:

\vspace{0.2cm}\noindent Gregory Duran Cunha\\ Instituto de Ciências Matemáticas e de Computação\\  Universidade de São Paulo\\ Avenida Trabalhador São-carlense, 400\\ 13566-590 - São Carlos - SP (Brazil).\\E--mail: {\tt gduran@usp.br } \\

\end{document}